\documentclass[11pt,a4paper]{article}
\usepackage[T1]{fontenc}
\usepackage{lmodern}
\usepackage{amsmath,amssymb,amsthm,mathtools}
\usepackage[margin=1in]{geometry}
\usepackage{microtype}
\usepackage{needspace}
\usepackage{fancyhdr}
\usepackage{hyperref}
\usepackage{aliascnt}
\usepackage[nameinlink,capitalise,noabbrev]{cleveref}
\usepackage{tikz}
\usetikzlibrary{arrows.meta,decorations.markings}
\hypersetup{
	colorlinks=false,
	pdfborder={0 0 1},
	linkbordercolor={1 0 0},
	citebordercolor={0 1 0}
}

\newcommand{\thmref}[1]{%
	\hyperref[#1]{Theorem~\ref*{#1}}%
}

\newtheorem{theorem}{Theorem}[section]
\newaliascnt{lemma}{theorem}
\newtheorem{lemma}[lemma]{Lemma}
\aliascntresetthe{lemma}
\newaliascnt{corollary}{theorem}
\newtheorem{corollary}[corollary]{Corollary}
\aliascntresetthe{corollary}
\newtheorem{proposition}{Proposition}[section]
\crefname{proposition}{Proposition}{Propositions}
\Crefname{proposition}{Proposition}{Propositions}
\newtheorem{claim}{Claim}
\crefname{claim}{Claim}{Claims}
\Crefname{claim}{Claim}{Claims}
\theoremstyle{remark}

\numberwithin{equation}{section}
\DeclareMathOperator{\ex}{ex}
\newcommand{\Prob}{\mathbb{P}}
\newcommand{\E}{\mathbb{E}}
\newcommand{\calP}{\mathcal{P}}
\newcommand{\calC}{\mathcal{C}}
\newcommand{\al}{\alpha_2}
\title{A near-linear Chv\'atal--Erd\H{o}s condition\\
	for Hamilton cycles in digraphs}
\author{
	Chengli Li\thanks{School of Mathematical Sciences, South China Normal University,
	Guangzhou 510631, P.R. China. Email: \texttt{lichengli0130@126.com}.}
	\and
	Bo Ning\thanks{College of Computer Science, Nankai University,
		Tianjin 300350, P.R. China. Email: \texttt{bo.ning@nankai.edu.cn}.
}
}

\date{}
\begin{document}
	\maketitle
	\begin{abstract}
		For a digraph $D$, let $\alpha_2(D)$ be the largest size of a vertex set
		containing no directed $2$-cycle. Let $f_2(a)$ be the least positive integer
		$k$ such that every $k$-strongly connected digraph $D$ with
		$\alpha_2(D)\le a$ has a Hamilton cycle. Jackson and Ordaz conjectured that
		$f_2(a)\le a+1$. Towards this conjecture, we establish the near-linear bound
		$
		f_2(a)=O\!\left(\frac{a(\log a)^4}{(\log\log a)^2}\right),
		$
		and hence
		$f_2(a)=O_\varepsilon(a^{1+\varepsilon})$ for every fixed $\varepsilon>0$.
		We also disprove the pancyclicity conjecture of Jackson and Ordaz
		that every digraph $D$ with $\kappa(D)\ge\alpha_2(D)+1$ contains
		a directed cycle of every length from $2$ to $|V(D)|$.
	\end{abstract}
	
	\section{Introduction}	
All graphs are finite and simple, and all digraphs are finite, without loops or parallel arcs.
Throughout this paper, we use $G$ for a graph and $D$ for a digraph.
A graph $G$ of order $n\ge3$ is \emph{Hamiltonian} if it contains a cycle through all its vertices, and \emph{pancyclic} if it contains cycles of every length from $3$ to $n$.
The notion of pancyclicity was introduced by Bondy in~\cite{BondyP}. In 1973, he proposed his famous meta-conjecture~\cite{BondyMeta}, which asserts that any non-trivial condition implying Hamiltonicity should also imply pancyclicity, up to a small class of exceptional graphs.

One such condition is the classical theorem of Chv\'atal and Erd\H{o}s~\cite{CE}.
For a graph $G$, let $\alpha(G)$ denote its independence number, the maximum size of a set of pairwise nonadjacent vertices, and let $\kappa(G)$ denote its vertex-connectivity.
The theorem states that every graph $G$ of order at least three with $\kappa(G)\ge\alpha(G)$ is Hamiltonian.
In the spirit of Bondy's meta-conjecture, Jackson and Ordaz~\cite{JO} conjectured that every graph $G$ with $\kappa(G)>\alpha(G)$ is pancyclic.
The strict inequality is necessary, since balanced complete bipartite graphs satisfy $\kappa(G)=\alpha(G)$ and have no odd cycles.

Keevash and Sudakov~\cite{KS} proved that $\kappa(G)\ge 600\alpha(G)$ implies pancyclicity, establishing the conjecture up to a constant factor. Dragani\'c, Munh\'a Correia and Sudakov~\cite{DMSCE} subsequently obtained an asymptotically optimal result, showing that
$
\kappa(G)\ge (1+o(1))\alpha(G)
$
already suffices for pancyclicity.
More recently, Letzter~\cite{Letzter} proved the conjecture for all sufficiently large graphs, retaining the condition $\kappa(G)>\alpha(G)$ without an asymptotic loss.
A closely related and important problem was posed by Erd\H{o}s~\cite{ErdosCycles} in 1972. He asked how large a Hamiltonian graph must be, in terms of its independence number, to guarantee pancyclicity, and conjectured that order $\Omega(\alpha(G)^2)$ should suffice. Dragani\'c, Munh\'a Correia and Sudakov~\cite{DMSHam} resolved this conjecture by proving that every Hamiltonian graph with independence number at most $k$ and at least $(2+o(1))k^2$ vertices is pancyclic, and this bound is asymptotically best possible.

Hamiltonicity and cycles of different lengths have also been extensively studied in digraphs; see the surveys~\cite{JO,KO}.
A digraph is \emph{Hamiltonian} if it contains a directed cycle through every vertex, and \emph{pancyclic} if it contains a directed cycle of every length from $2$ to its order.
A nonempty digraph is \emph{strongly connected} if it contains a directed path from $u$ to $v$ for every ordered pair of vertices $u,v$, allowing a path of length zero when $u=v$.
Ghouila-Houri~\cite{GH} proved that a strongly connected $n$-vertex digraph is Hamiltonian if the sum of the indegree and outdegree of each vertex is at least $n$.
H\"aggkvist and Thomassen~\cite{HTPancyclic} showed that the same degree condition also guarantees pancyclicity, with the complete balanced bipartite digraph having both arcs between every pair of vertices in different parts as the only possible exception.

Sharp degree conditions have also been obtained for oriented graphs, that is, digraphs with no pair of oppositely directed arcs.
Keevash, K\"uhn and Osthus~\cite{KKO} proved that every sufficiently large $n$-vertex oriented graph with minimum indegree and outdegree at least $(3n-4)/8$ is Hamiltonian, and that this bound is best possible.
More recently, Lo, Patel and Y\i ld\i z~\cite{LPY} proved that every sufficiently large $d$-regular oriented graph on $n\le4d+1$ vertices is Hamiltonian, where $d$-regular means that each vertex has indegree and outdegree $d$.

The Chv\'atal--Erd\H{o}s theorem suggests seeking analogous conditions in terms of connectivity and independence in digraphs.
For a positive integer $k$, a digraph $D$ is \emph{$k$-strongly connected} if $|V(D)|\ge k+1$ and $D-S$ is strongly connected for every $S\subseteq V(D)$ with $|S|<k$.
We write $\kappa(D)$ for the largest such $k$, and set $\kappa(D)=0$ if $D$ is not strongly connected or has at most one vertex.
For a digraph $D$, its \emph{$2$-cycle graph} $G_2(D)$ is the graph on $V(D)$ in which two vertices are adjacent precisely when they form a directed $2$-cycle in $D$.
We define $\al(D)=\alpha(G_2(D))$; equivalently, $\al(D)$ is the largest size of a vertex set containing no directed $2$-cycle.
Replacing every edge of a graph by two opposite arcs preserves its connectivity and turns its independence number into $\al(D)$.
Thus $\al$ gives a natural directed analogue of the independence number in Chv\'atal--Erd\H{o}s type conditions; see~\cite{BG,JO,KO}.

For a positive integer $a$, let $f_2(a)$ be the least positive integer $k$ such that every $k$-strongly connected digraph $D$ with $\al(D)\le a$ is Hamiltonian.
Jackson and Ordaz~\cite{JO} conjectured that $\kappa(D)\ge\al(D)+1$ guarantees a Hamilton cycle, which would imply $f_2(a)\le a+1$.
Jackson~\cite{Jackson} proved that $f_2$ exists and obtained $f_2(a)\le2^a(a+2)!$.
K\"uhn and Osthus~\cite{KO} asked whether $f_2$ admits a polynomial upper bound.
Ai and Shi~\cite{AiShi} recently proved that $f_2(a)\le 2a^3+2$.
On the other hand, replacing every edge of $K_{a-1,a}$ by two opposite arcs gives a digraph $D$ with $\kappa(D)=a-1$ and $\al(D)=a$ that is not Hamiltonian, since every directed cycle meets the two parts in the same number of vertices.
Hence $f_2(a)\ge a$ for $a\ge2$.
We prove a near-linear upper bound, within a polylogarithmic factor of the bound conjectured by Jackson and Ordaz.
	\begin{theorem}\label{thm:asymptotic}
		There is an absolute constant $c>0$ such that, for every integer $a\ge3$,
		\[
		f_2(a)\le c\frac{a(\log a)^4}{(\log\log a)^2}.
		\]
	\end{theorem}
	All unsubscripted logarithms are natural.
	We first prove a more general bound depending explicitly on the length of the even cycles used in the proof.
	For integers $a\ge2$ and $k\ge2$, let $\ell=\lceil\log_2 a\rceil+1$ and
	\[
	r_k(a)=\left\lceil\left(2(k-1)a\right)^{\frac{k}{k-1}}
	+\left(32k+\frac{k}{k-1}\right)a\right\rceil,
	\qquad h_k(a)=2^{18}k^2a\ell^2.
	\]
	\begin{theorem}\label{thm:quantitative}
		For all integers $a\ge2$ and $k\ge2$ with $k\not\equiv0\pmod{3}$, we have $f_2(a)\le r_k(a)+h_k(a)$.
	\end{theorem}
	From Theorem~\ref{thm:asymptotic}, we obtain the following corollary.
	\begin{corollary}\label{cor:epsilon}
		For every fixed $\varepsilon>0$, we have $f_2(a)=O_\varepsilon(a^{1+\varepsilon})$.
	\end{corollary}
	
	Jackson and Ordaz~\cite{JO} also conjectured
	that every digraph $D$ with $\kappa(D)\ge\al(D)+1$ is pancyclic.
	We disprove this conjecture. In fact, even a linear connectivity
	condition in terms of $\al(D)$ does not guarantee pancyclicity.
	
	\begin{theorem}\label{thm:pancyclic-obstruction}
		There is no constant $c>0$ such that every Hamiltonian digraph $D$ satisfying
		$\kappa(D)\ge c\al(D)$ is pancyclic.
	\end{theorem}
\medskip\noindent\textbf{Proof sketch of \Cref{thm:asymptotic}.}
We first prove \Cref{thm:quantitative}.
Let $D$ satisfy $\alpha_2(D)\le a$ and
$\kappa(D)\ge r_k(a)+h_k(a)$, and let $G=G_2(D)$.
Choose a maximal family $\calC$ of pairwise vertex-disjoint
copies of $C_{2k}$ in $G$.
Let $R=V(G)\setminus\bigcup_{C\in\calC}V(C)$ be the remaining vertex set.
Then $G[R]$ is $C_{2k}$-free.
Motivated by the work of Chen, Faudree and Gould~\cite{CFG}, we apply Pikhurko's even-cycle Tur\'an bound and the Caro--Wei inequality to obtain $|R|\le r_k(a)$.
Hence $D-R$ is $h_k(a)$-strongly connected.

For each $C\in\calC$, we choose a retained edge $e_C$ and let $P_C$ be the complementary path on the remaining $2k-2$ vertices.
The main ingredient is \Cref{lem:selection}, which allows us to choose these edges, together with a direction and a colour for each $P_C$, so that the corresponding auxiliary digraph $D_j$ is $a$-strongly connected for every colour $j$ and has the required degree properties.
A directed path in $D_j$ expands to a directed path in $D$ by concatenating whole complementary paths of colour $j$.
The proof uses Chernoff bounds and the Caro--Wei inequality for the degree estimates, and combines strong connectivity with negative association to rule out directed separations of order less than $a$ in $D_j$.

The vertices of $R$, together with the retained edges, can be covered by at most $a$ pairwise vertex-disjoint directed paths.
We now use one new colour in each round.
By \Cref{lem:linkage}, we join the current paths in pairs, reducing their number from $p$ to $\lceil p/2\rceil$.
After at most $\lceil\log_2 a\rceil$ rounds, one path remains, and we use a final colour to close it into a directed cycle $C^*$.
This cycle covers $R$ and every retained edge, and uses each complementary path either in its entirety or not at all.
For each complementary path $P_C$ still outside $C^*$, we replace the retained edge $e_C$ by the other directed path around $C$ between its endpoints.
These replacements give a Hamilton cycle and prove \Cref{thm:quantitative}.
Finally, \Cref{thm:asymptotic} follows from \Cref{thm:quantitative} by optimizing the choice of $k$.

\medskip\noindent\textbf{Organization.}
In Section~\ref{sec:preliminaries}, we introduce the notation and collect the preliminary results.
In Section~\ref{sec:selection}, we prove the path-selection lemma.
In Section~\ref{sec:proofs}, we prove \Cref{thm:quantitative} and derive \Cref{thm:asymptotic}.
In Section~\ref{sec:pancyclicity}, we construct the counterexamples to the Jackson--Ordaz pancyclicity conjecture.

\medskip\noindent\textbf{Declaration of AI Use.}
The authors are grateful to AI tools for valuable assistance in exploring some of the proof strategies and constructing examples. This project began in July 2026. Early discussions with ChatGPT 5.6 explored several approaches, including linear expanders and the regularity lemma, but yielded only partial results. The authors compiled their reading notes and records of earlier attempts and shared them with the AI tools as background for subsequent discussions. The main progress came on 9 September 2026, when extensive author-guided discussions with ChatGPT 6 led to the near-linear bound stated in \Cref{thm:asymptotic}.
Starting from an initial ten-page draft, the authors completely rewrote the Introduction and reorganized and partly expanded the proofs, producing the present seventeen-page manuscript. The authors directed the research process, independently verified all mathematical arguments and calculations, and take full responsibility for the mathematical correctness and exposition of the paper.

	\section{Preliminaries}\label{sec:preliminaries}
	We follow the graph and digraph terminology of~\cite{BG,BM}.
	We write $V(G)$ and $E(G)$ for the vertex and edge sets of a graph, and $V(D)$ and $A(D)$ for the vertex and arc sets of a digraph.
	The order of a graph or digraph is its number of vertices, and its size is its number of edges or arcs.
	A directed path visits distinct vertices in an order consistent with its arcs, a directed cycle does the same cyclically and has length at least two.
	For $U\subseteq V(D)$, the induced subdigraph $D[U]$ keeps exactly the vertices in $U$ and all arcs
	of $D$ with both endpoints in $U$, and $D-U=D[V(D)\setminus U]$; induced subgraphs of an undirected graph
	are defined in the same way.
	For $v\in V(D)$, the outneighbourhood and inneighbourhood are
	\[
	N_D^+(v)=\{w\in V(D):vw\in A(D)\},\qquad
	N_D^-(v)=\{w\in V(D):wv\in A(D)\}.
	\]
	
	Let
	$
	d_D^+(v)=|N_D^+(v)|
	$
	for $v\in V(D)$, and, when $D$ is nonempty, let
	$
	\delta^+(D)=\min_{v\in V(D)} d_D^+(v).
	$
	For $v\in V(D)$ and $U\subseteq V(D)$, let
	$
	d_U^+(v)=|N_D^+(v)\cap U|.
	$
	The corresponding notation $d_D^-(v)$, $\delta^-(D)$, and
	$d_U^-(v)$ is defined using inneighbours. The notation $d_U^+(v)$ and
	$d_U^-(v)$ does not require $v\in U$.
	The tail and head of an arc $uv$ are $u$ and $v$, respectively.
	A source is a vertex of indegree zero, and a sink is a vertex of outdegree zero.
	For an undirected graph $G$, let
	$N_G(v)=\{w:vw\in E(G)\},$
	$N_G[v]=N_G(v)\cup\{v\}$ and
	$d_G(v)=|N_G(v)|.$
	
	The length of a path or cycle is its number of edges or arcs, and the internal vertices of a path
	are all its vertices other than its endpoints.
	A \emph{path cover} of a graph or digraph $H$ is a family of pairwise vertex-disjoint paths
	whose vertex sets partition $V(H)$; for a digraph, the paths are directed.
	A \emph{linear forest} is a graph whose components are paths, including isolated
	vertices, and it is \emph{spanning} in $G$ if it has vertex set $V(G)$.
	The endpoints of a linear forest are the vertices of degree at most one, and every path in $G_2(D)$ can
	be directed in either order in $D$.
	For a graph $H$, we write $\ex(n,H)$ for the largest number of
	edges in an $n$-vertex graph containing no subgraph isomorphic to $H$, not necessarily an induced
	subgraph.
	We use $[s]=\{1,\ldots,s\}$ and $A\mathbin{\dot\cup}B$ for a disjoint union.
	
	If $D$ is $k$-strongly connected, then every vertex has at least $k$ inneighbours and $k$ outneighbours.
	Indeed, if a vertex had fewer than $k$ outneighbours, deleting them would leave it unable to reach another remaining vertex, which exists because $|V(D)|\ge k+1$. The inneighbour assertion is analogous.
	Also, deleting a set $S$ with $|S|=s<k$ leaves a $(k-s)$-strongly connected digraph, because any
	further deletion of fewer than $k-s$ vertices removes fewer than $k$ vertices in total, and at least
	$k-s+1$ vertices remain before that further deletion.
	
	\begin{theorem}[Caro--Wei~\cite{Caro,Wei}]\label{thm:carowei}
		Every nonempty graph $G$ of order $n$ and size $m$ satisfies
		\[
		\sum_{v\in V(G)}\frac{1}{d_G(v)+1}\le\alpha(G),
		\qquad \frac{n^2}{2m+n}\le\alpha(G).
		\]
	\end{theorem}
	\begin{theorem}[{Pikhurko~\cite{Pikhurko}}]\label{thm:pikhurko}
		For all integers $k\ge2$ and $n\ge1$, we have
		\[
		\ex(n,C_{2k})\le(k-1)n^{1+\frac{1}{k}}+16(k-1)n.
		\]
	\end{theorem}
	\Needspace{6\baselineskip}
	\begin{lemma}\label{lem:remainder}
		Let $a\ge2$ and $k\ge2$ be integers.
		If a graph $G$ contains no $C_{2k}$ and satisfies $\alpha(G)\le a$, then $|V(G)|\le r_k(a)$.
	\end{lemma}
	\begin{proof}
		Let $n=|V(G)|$. If $n=0$, the conclusion is immediate.
		Hence assume that $n\ge1$, and let $m=|E(G)|$.
		By \Cref{thm:carowei}, $n^2\le a(2m+n)$, while \Cref{thm:pikhurko} gives
		$2m\le2(k-1)n^{1+\frac{1}{k}}+32(k-1)n$.
		Substituting the second bound into the first and dividing by $n$ yields
		\begin{equation}\label{eq:remainder-start}
			n\le 2(k-1)a n^{\frac{1}{k}}+(32(k-1)+1)a.
		\end{equation}
		Let $x_0=(2(k-1)a)^{\frac{k}{k-1}}$. Then $2(k-1)a\,x_0^{\frac{1}{k}}=x_0$.
		If $n\le x_0$, the conclusion follows. Otherwise, let $t =\frac{n-x_0}{x_0}>0$, so $n=x_0(1+t)$.
		The binomial theorem gives $(1+\frac{t }{k})^k\ge1+t $, and hence $(1+t )^{\frac{1}{k}}\le1+\frac{t }{k}$.
		Consequently, $2(k-1)a n^{\frac{1}{k}}\le x_0(1+\frac{t }{k})=x_0+\frac{n-x_0}{k}$.
		Using this in \eqref{eq:remainder-start}, we obtain
		$\frac{k-1}{k}(n-x_0)\le (32(k-1)+1)a$ and so
		\[
		n\le x_0+\left(32k+\frac{k}{k-1}\right)a\le r_k(a).
		\]
		This proves the lemma.
	\end{proof}
	
	Gallai and Milgram~\cite{GM} proved that every digraph has a path cover
	of size at most the independence number of its underlying graph.
	For undirected graphs, the corresponding bound follows from the standard
	endpoint-joining argument: in a path cover with the minimum number of
	components, endpoints belonging to distinct paths are nonadjacent; see,
	for example, Han~\cite{Han}. We shall use the following version, which
	keeps a prescribed spanning linear forest.
	\begin{lemma}\label{lem:forest}
		Let $G$ be a graph and let $F$ be a spanning linear forest of $G$.
		There is a spanning linear forest $F'$ of $G$ containing every edge of $F$, with at most $\alpha(G)$
		components, such that every endpoint of $F'$ is an endpoint of $F$ and every component of $F$ is a subpath of
		a component of $F'$.
	\end{lemma}
	\begin{proof}
		Starting from $F$, repeatedly join two distinct components whenever an endpoint of one is adjacent in $G$
		to an endpoint of the other.
		The union of the two paths and this joining edge is again a path, including when either path is a
		single vertex.
		Thus every operation preserves all existing edges, reduces the number of components by one, and
		creates no new endpoint.
		Since the graph is finite, the process stops at a spanning linear forest $F'$ whose endpoints in
		different components are pairwise nonadjacent.
		Choosing one endpoint from each component gives an independent set, so there are at most $\alpha(G)$
		components.
	\end{proof}
	A \emph{matching} is a set of edges with pairwise disjoint endpoints, and a \emph{vertex cover} is a
	set of vertices meeting every edge.
	
	\begin{theorem}[K\H{o}nig's theorem; see~\cite{BM}]
		\label{thm:konig}
		Let $G$ be a bipartite graph. Then the maximum size of a matching
		in $G$ equals the minimum size of a vertex cover of $G$.
	\end{theorem}
	
	For disjoint vertex sets $A,B$ of a digraph, a \emph{matching of arcs from $A$ to $B$} consists of
	arcs with tails in $A$, heads in $B$, and pairwise disjoint endpoints.
	\begin{lemma}\label{lem:cutmatching}
		Let $D$ be a $k$-strongly connected digraph, where $k\ge1$, and let $V(D)=A\mathbin{\dot\cup}B$ with
		$A,B\ne\emptyset$.
		Then $D$ has a matching of arcs from $A$ to $B$ of size at least $\min\{k,|A|,|B|\}$.
	\end{lemma}
	\begin{proof}
		Form the bipartite graph $H$ with parts $A,B$ in which $uv$ is an edge precisely when $u\in A$,
		$v\in B$, and $uv\in A(D)$.
		If its maximum matching had size
		$m<\min\{k,|A|,|B|\}$, \Cref{thm:konig} would give a
		vertex cover $S$ of size $m$.
		Since $|S|<|A|,|B|$, both $A\setminus S$ and $B\setminus S$ are nonempty, and the cover property
		says that no arc goes from the former to the latter in $D-S$.
		Any directed path from $A\setminus S$ to $B\setminus S$ would have a first arc crossing this
		partition, so no such path exists.
		Thus $D-S$ is not strongly connected, contradicting $|S|<k$.
	\end{proof}
	
	We write $\Prob$ for probability and $\E$ for expectation,
	while $A^c$ denotes the complementary event of $A$.
	We shall use Markov's inequality in the form
	$
	\Prob[X\ge t]\le\frac{\E[X]}{t}
	$
	for every nonnegative random variable $X$ and every $t>0$.
	A Bernoulli random variable takes values $0$ and $1$.
	
	A finite random vector $(Z_i)_{i\in I}$ is
	\emph{negatively associated} if, for disjoint subsets
	$I_1,I_2\subseteq I$ and coordinatewise nondecreasing
	functions $f,g$,
	\[
	\E\!\left[
	f\bigl((Z_i)_{i\in I_1}\bigr)
	g\bigl((Z_i)_{i\in I_2}\bigr)
	\right]
	\le
	\E\!\left[f\bigl((Z_i)_{i\in I_1}\bigr)\right]\,
	\E\!\left[g\bigl((Z_i)_{i\in I_2}\bigr)\right],
	\]
	whenever the expectations exist.
	Here coordinatewise nondecreasing means that increasing
	a coordinate while fixing the others cannot decrease
	the function value.
	The following three standard properties of negative
	association are due to Joag-Dev and Proschan~\cite{JDP}.
	See also Dubhashi and Ranjan~\cite{DR}.
	
	\begin{proposition}[Joag-Dev and Proschan~\cite{JDP}]
		\label{prop:negative-association}\leavevmode
		\begin{enumerate}
			\renewcommand{\labelenumi}{\textup{(\roman{enumi})}}
			\item A finite $0$--$1$ random vector with at most
			one coordinate equal to $1$ is negatively associated.
			
			\item The concatenation of finitely many mutually
			independent, negatively associated random vectors
			is negatively associated.
			
			\item Let $(Z_i)_{i\in I}$ be negatively associated,
			let $I_1,\ldots,I_t\subseteq I$ be pairwise disjoint,
			and let $f_1,\ldots,f_t$ be coordinatewise
			nondecreasing functions.
			Then the random vector
			$
			\bigl(f_j((Z_i)_{i\in I_j})\bigr)_{j\in[t]}
			$
			is negatively associated.
		\end{enumerate}
	\end{proposition}
	
	If $Z$ is a sum of independent or negatively associated
	Bernoulli random variables with mean $\mu=\E[Z]$, then
	the Chernoff bound gives
	$
	\Prob\!\left[Z\le\frac{\mu}{2}\right]\le e^{-\mu/8}.
	$
	See~\cite{AS} for the independent case and~\cite{DR}
	for the negatively associated case.
	The next lemma applies this bound together with
	\Cref{prop:negative-association}.
	
	\Needspace{9\baselineskip}
	\begin{lemma}\label{lem:part-selection}
		Let $M=\{u_iv_i:i\in[r]\}$ be a matching from a set $A$
		to a disjoint set $B$, and partition $A$ and $B$
		separately into parts.
		In each part select at most one vertex, independently
		of all other parts, and let $p_i,q_i$ be the probabilities
		that $u_i,v_i$ are selected, respectively.
		Call an edge $u_iv_i\in M$ \emph{surviving} if both
		$u_i$ and $v_i$ are selected.
		If $Y$ is the number of surviving edges and
		$\mu=\sum_{i=1}^r p_iq_i$, then
		\[
		\Prob\!\left[Y\le\frac{\mu}{2}\right]\le e^{-\mu/8}.
		\]
		Moreover, the surviving edges meet distinct parts
		on each side.
	\end{lemma}
	
	\begin{proof}
		Let $X_i$ indicate that $u_i$ is selected and let $Z_i$
		indicate that $v_i$ is selected.
		Within each part the full vector of vertex-selection
		indicators has at most one $1$, so it is negatively
		associated by
		\Cref{prop:negative-association}\textup{(i)}.
		Since the selections in different parts are independent,
		the full vector of all selection indicators is negatively
		associated by
		\Cref{prop:negative-association}\textup{(ii)}.
		
		For each $i$, let $Y_i$ be the indicator that $u_iv_i$
		survives. Then $Y_i=X_iZ_i$.
		Because $M$ is a matching, the coordinate pairs
		$\{X_i,Z_i\}$ are disjoint for distinct $i$, although
		different coordinates may correspond to vertices in
		the same part.
		Each product is nondecreasing in its two coordinates,
		so \Cref{prop:negative-association}\textup{(iii)}
		implies that $(Y_i)_{i\in[r]}$ is negatively associated.
		
		Since $u_i$ and $v_i$ belong to distinct parts,
		their selections are independent.
		Thus $\E[Y_i]=p_iq_i$ and
		$\E[Y]=\sum_{i=1}^r p_iq_i=\mu$.
		The Chernoff bound therefore gives
		$
		\Prob\!\left[Y\le\frac{\mu}{2}\right]\le e^{-\mu/8}.
		$
		Finally, the surviving edges meet distinct parts
		on each side, since each part supplies at most
		one selected vertex.
	\end{proof}
	
	\section{A path-selection lemma}\label{sec:selection}
	Let $a\ge2$ and $t\ge4$ be integers with $t\not\equiv0\pmod{3}$, let $1\le\ell\le2a$ be an integer, and let
	\[
	h=2^{16}t^2a\ell^2,\qquad d_0=\frac{h}{4t\ell}.
	\]
	Let $D$ satisfy $\al(D)\le a$, choose $R\subseteq V(D)$, and suppose that $D'=D-R$ is $h$-strongly
	connected, and let $U=V(D')$.
	Let $T\subseteq R$ be fixed before any random choices, with $|T|\le2a$ and $d_U^+(v),d_U^-(v)\ge h$
	for every $v\in T$.
	Assume that a family $\calC$ of vertex-disjoint $t$-cycles in $G_2(D)$ covers exactly $V(D')$.
	
	For each $C\in\calC$, choose an edge $e_C\in E(C)$, called the \emph{retained edge}, and delete its
	two endpoints from the cycle $C$ to obtain a path $P_C$ on $t-2$ vertices,
	called its \emph{complementary path}.
	Choose one of the two directions of $P_C$ and a colour in $[\ell]$, with choices independent between
	different cycles.
	For each colour $j$, let $\calP_j$ be its family of oriented paths, and let $x_P$ and $y_P$ be the
	first and last vertices of an oriented path $P$, respectively.
	Let $D_j$ be the auxiliary digraph whose vertices are $v_P$, one for each $P\in\calP_j$, with arcs defined by
	\[
	v_Pv_{P'}\in A(D_j)
	\quad\Longleftrightarrow\quad
	P\ne P'\ \hbox{and}\ y_Px_{P'}\in A(D).
	\]
	A directed path of $D_j$ expands to a directed path of $D$ by traversing each represented path in full and, for each pair of consecutive represented paths $P,P'$, using the arc $y_Px_{P'}$ of $D$.
	
	For $v\in V(D')\cup T$, let $d_j^+(v)$ count the paths $P\in\calP_j$ with $vx_P\in A(D)$, and let
	$d_j^-(v)$ count those with $y_Pv\in A(D)$.
	When $v\in V(D')$, let $C$ be the unique cycle in $\calC$ containing $v$, and exclude the path arising from $C$ from these counts. When $v\in T$, no such exclusion is made.
	\begin{lemma}\label{lem:selection}
		With the notation and assumptions above, the retained edges,
		directions, and colours can be chosen such that, for every
		$j\in[\ell]$,
		\[
		\kappa(D_j)\ge a,\qquad
		d_j^+(v),d_j^-(v)\ge a
		\quad\text{for all }v\in V(D')\cup T.
		\]
	\end{lemma}
	\begin{proof}
		Let $G=G_2(D)$. Independently for each $C\in\calC$, choose its retained edge,
		path direction, and colour uniformly from the $2t\ell$ possibilities, and denote
		this choice by $\xi_C$. For $\sigma\in\{+,-\}$, we use $d_U^\sigma(v)$ and
		$d_j^\sigma(v)$ for the corresponding $+$ or $-$ quantities defined above.
		\begin{claim}\label{cl:neighbours}
			With probability greater than $\frac{15}{16}$, every $d_j^\sigma(v)$, for $j\in[\ell]$,
			$v\in V(D')\cup T$, and $\sigma\in\{+,-\}$, is at least $d_0$.
		\end{claim}
		\begin{proof}
			Fix $v\in V(D')\cup T$, $j\in[\ell]$, and $\sigma\in\{+,-\}$, and let
			$d=d_U^\sigma(v)\ge h$. A fixed vertex $z$ of a cycle $C$ is an endpoint of $P_C$
			for exactly two choices of the retained edge. For each resulting path, exactly one
			orientation makes $z$ its first vertex and exactly one makes $z$ its last vertex.
			Thus $z$ is the relevant endpoint of a colour-$j$ path with probability $\frac{1}{t\ell}$,
			where the relevant endpoint is the first vertex when $\sigma=+$ and the last vertex
			when $\sigma=-$. If $C$ does not contain $v$ and contains $s$ outneighbours of
			$v$ when $\sigma=+$, or $s$ inneighbours when $\sigma=-$, its contribution to
			$d_j^\sigma(v)$ is a Bernoulli variable of mean $\frac{s}{t\ell}$.
			Omitting the cycle containing $v$, when necessary, loses at most $t$ such
			neighbours. The remaining contributions are independent, with total mean
			$\mu\ge\frac{d-t}{t\ell}\ge \frac{d}{2t\ell}$, since $d\ge h=2^{16}t^2a\ell^2\ge2t$.
			Since $\frac{d}{4t\ell}\le\frac{\mu}{2}$, the Chernoff bound gives
			\begin{equation}\label{eq:vertex-tail}
				\Prob\!\left[d_j^\sigma(v)<\frac{d}{4t\ell}\right]\le e^{-\frac{d}{16t\ell}}.
			\end{equation}
			
			We next sum these bounds over $V(D')$ using the Caro--Wei inequality.
			Let $\beta =16t\ell$. Every neighbour of $v$ in $G[V(D')]$ is both an
			inneighbour and an outneighbour of $v$ in $D'$, so
			\Cref{thm:carowei} gives
			\[
			\sum_{v\in V(D')}\frac{1}{d_U^\sigma(v)+1}
			\le
			\sum_{v\in V(D')}\frac{1}{d_{G[V(D')]}(v)+1}
			\le a.
			\]
			Since $h=2^{16}t^2a\ell^2\ge16t\ell=\beta $, the inequality $1+x\le e^x$ gives, for $d\ge h$,
			\[
			d+1=(h+1)\left(1+\frac{d-h}{h+1}\right)
			\le(h+1)e^{\frac{d-h}{\beta }}.
			\]
			Consequently,
			\[
			\sum_{v\in V(D')}e^{-\frac{d_U^\sigma(v)}{\beta }}
			\le(h+1)e^{-\frac{h}{\beta }}\sum_{v\in V(D')}\frac{1}{d_U^\sigma(v)+1}
			\le a(h+1)e^{-\frac{h}{\beta }}.
			\]
			Let $E_0$ be the event in the claim, and let
			$x=\frac{h}{\beta}=2^{12}ta\ell\ge2^{15}$.
			Since $d_U^\sigma(v)\ge h$, we have
			$d_0\le \frac{d_U^\sigma(v)}{4t\ell}$.
			For each fixed $j\in[\ell]$ and $\sigma\in\{+,-\}$,
			combining \eqref{eq:vertex-tail} with the above bound gives
			\[
			\sum_{v\in V(D')}\Prob[d_j^\sigma(v)<d_0]
			\le a(h+1)e^{-x}.
			\]
			Each vertex of $T$ contributes at most $e^{-x}$, so
			$|T|\le2a$ and a union bound over all colours and both signs yield
			\[
			\begin{aligned}
				\Prob[E_0^c]
				\le2\ell\bigl(a(h+1)+|T|\bigr)e^{-x}
				\le2a\ell(h+3)e^{-x}.
			\end{aligned}
			\]
			Since $h\ge3$ and $h=16t\ell x$, we have
			\[
			2a\ell(h+3)
			\le4a\ell h
			=64at\ell^2x
			\le64x^3,
			\]
			where the last inequality uses $at\ell\le x$ and $\ell\le x$.
			Moreover, $x\ge2^{15}>64\cdot384$ implies
			$
			e^{x/2}\ge\frac{(x/2)^4}{4!}
			=\frac{x^4}{384}\ge64x^3.
			$
			Therefore,
			\begin{equation}\label{eq:neighbour-failure}
				\Prob[E_0^c]\le64x^3e^{-x}\le e^{-x/2}<\frac{1}{16}.
			\end{equation}
			This proves the claim.
		\end{proof}
		
		On $E_0$, every $D_j$ is nonempty, since $D'\ne\emptyset$ and $d_0>0$.
		For $v_P\in V(D_j)$, its outdegree is $d_j^+(y_P)$ and its indegree is
		$d_j^-(x_P)$, since the omitted cycle is precisely the cycle supplying $P$.
		Since $d_0=\frac{h}{4t\ell}=2^{14}ta\ell>a$, we have $|V(D_j)|\ge d_0+1>a$.
		It remains to rule out a partition
		$V(D_j)=L'\mathbin{\dot\cup}X\mathbin{\dot\cup}R'$ with $L',R'\ne\emptyset$,
		$|X|<a$, and no arc from $L'$ to $R'$; call this a \emph{bad separation}.
		Indeed, if $D_j-X$ is not strongly connected, the vertices reachable from a
		vertex that cannot reach all of $D_j-X$ form a possible set $L'$, and the
		remaining vertices form $R'$.
		
		Fix a colour $j$ and an ordering of $\calC$.
		We examine the choices $\xi_C$ one cycle at a time, following the fixed ordering. Each choice is fixed in advance and is not changed when its cycle is examined.
		To describe a possible bad separation, we build disjoint sets $I_L,I_R$, initially
		empty, whose union is independent in $G$.
		The vertices in $I_L$ will be last vertices $y_P$ of paths $P$ with
		$v_P\in L'$, and those in $I_R$ will be first vertices $x_P$ of paths $P$
		with $v_P\in R'$.
		For the current sets $I_L,I_R$, let
		\[
		\begin{split}
			A_L&=\{v\in V(G):N_G[v]\cap I_L\ne\emptyset,\ N_G[v]\cap I_R=\emptyset\},\\
			A_R&=\{v\in V(G):N_G[v]\cap I_R\ne\emptyset,\ N_G[v]\cap I_L=\emptyset\}.
		\end{split}
		\]
		The sets $A_L,A_R$ are disjoint and are redefined whenever $I_L$ or $I_R$ changes.
		We examine only a cycle not previously examined and contained in neither $A_L$
		nor $A_R$. A cycle contained in one of these sets at an earlier stage may
		therefore become available for examination later, as the sets change.
		For the complementary path $P$ of a cycle being examined, the choices that
		we need to record are those satisfying
		\begin{equation}\label{eq:record-condition}
			P\in\calP_j,\qquad x_P\notin A_R,\qquad y_P\notin A_L.
		\end{equation}
		\begin{claim}\label{cl:effective}
			If a cycle $C$ is contained in neither $A_L$ nor $A_R$, at least one of its
			$2t\ell$ choices satisfies \eqref{eq:record-condition}.
		\end{claim}
		
		\begin{proof}
			Let $C=v_0v_1\cdots v_{t-1}v_0$, with indices taken modulo $t$.
			The possible endpoint pairs of a complementary path are precisely
			$\{v_i,v_{i+3}\}$ for $i\in\mathbb Z_t$.
			Let $H$ be the graph on $V(C)$ whose edges are these pairs.
			Since $t\not\equiv0\pmod{3}$, we have $\gcd(3,t)=1$, so $H$ is connected.
			
			Suppose that no complementary path can be oriented with its first
			vertex outside $A_R$ and its last vertex outside $A_L$.
			For each edge $zw\in E(H)$, both orientations then fail these
			requirements. As $A_L\cap A_R=\emptyset$, this implies that
			$\{z,w\}\subseteq A_L$ or $\{z,w\}\subseteq A_R$.
			The connectedness of $H$ now implies that $V(C)$ is contained
			in $A_L$ or in $A_R$, a contradiction.
			Choose a suitable orientation and assign colour $j$.
			This gives \eqref{eq:record-condition}.
		\end{proof}
		
		Let $r_0=2a-1$ and let $N_0=\frac{h}{16t}=2^{12}ta\ell^2$.
		A word of length $r_0$ over $\{\mathrm{L},\mathrm{R},\mathrm{X}\}$ will specify how
		to proceed whenever \eqref{eq:record-condition} holds.
		Fix such a word, start with $I_L=I_R=\emptyset$ and a counter $\eta_X=0$, and always
		examine the first available cycle in the fixed ordering that is contained in
		neither $A_L$ nor $A_R$.
		If its path fails \eqref{eq:record-condition}, make no change.
		Otherwise read the next symbol of the word: $\mathrm{L}$ adds $y_P$ to $I_L$,
		$\mathrm{R}$ adds $x_P$ to $I_R$, and $\mathrm{X}$ increases $\eta_X$ by one.
		Discard the word if the indicated addition destroys independence, if $\eta_X=a$, or
		if a further symbol is needed after all $r_0$ symbols have been used.
		After each examination, apply the preceding rules. Unless the word is
		discarded, repeat the procedure until either $N_0$ cycles have been examined
		or every unexamined cycle is contained in $A_L$ or in $A_R$.
		The next cycle and every decision in this procedure depend only on the fixed
		word and the choices on cycles already examined.
		
		A bad separation \emph{agrees} with the procedure if, whenever a symbol is
		used for a path $P$, that symbol is $\mathrm{L}$, $\mathrm{R}$, or $\mathrm{X}$
		according as $v_P\in L'$, $v_P\in R'$, or $v_P\in X$.
		
		\begin{claim}\label{cl:records}
			Every bad separation agrees with the procedure for some word that is not
			discarded. There are $3^{2a-1}$ possible words.
		\end{claim}
		
		\begin{proof}
			Fix a bad separation. Whenever \eqref{eq:record-condition} holds for a path
			$P$, choose the next symbol to be $\mathrm{L}$, $\mathrm{R}$, or $\mathrm{X}$
			according as $v_P\in L'$, $v_P\in R'$, or $v_P\in X$.
			
			Suppose first that $v_P\in L'$, and let $y=y_P$. Every vertex of $I_R$ is
			the first vertex of an earlier path represented in $R'$. An edge of $G$
			joining $y$ to such a vertex would give an arc from $L'$ to $R'$ in $D_j$,
			contrary to the definition of a bad separation. Thus $y$ has no neighbour in $I_R$.
			Moreover, since the current cycle has not been examined before and the cycles
			of $\calC$ are vertex-disjoint, we have $y\notin I_R$. Hence
			$N_G[y]\cap I_R=\emptyset$. Together with $y\notin A_L$, this implies
			$N_G[y]\cap I_L=\emptyset$. Thus adding $y$ to $I_L$ preserves independence.
			
			Similarly, suppose that $v_P\in R'$, and let $x=x_P$. If $x$ were adjacent
			in $G$ to a vertex of $I_L$, the corresponding arc would go from $L'$ to
			$R'$ in $D_j$, again a contradiction. Since the current cycle has not been examined before,
			$x\notin I_L$, and hence $N_G[x]\cap I_L=\emptyset$. Together with
			$x\notin A_R$, this gives $N_G[x]\cap I_R=\emptyset$. Thus adding $x$ to
			$I_R$ also preserves independence.
			
			Each symbol $\mathrm{L}$ or $\mathrm{R}$ therefore adds a new vertex to the
			independent set $I_L\cup I_R$, so there are at most $a$ such symbols.
			Each symbol $\mathrm{X}$ corresponds to a different vertex $v_P\in X$, so
			there are at most $a-1$ such symbols. Hence at most
			$a+(a-1)=2a-1=r_0$ symbols are used, and the word is never discarded.
			If fewer than $r_0$ symbols are used, fill the remaining positions with
			$\mathrm{X}$. These additional symbols are never read and do not affect
			the procedure. Since each of the $r_0=2a-1$ positions has three possible
			symbols, there are $3^{2a-1}$ possible words.
		\end{proof}
		
		\begin{claim}\label{cl:overflow}
			For a fixed word, the probability that $N_0$ cycles are examined without discarding
			the word is at most $e^{-\frac{h}{256t^2\ell}}$.
		\end{claim}
		\begin{proof}
			The next cycle is determined only by the choices on the
			cycles already examined, so each of its $2t\ell$ possible
			choices still has conditional probability $\frac{1}{2t\ell}$.
			By \Cref{cl:effective}, the conditional probability that
			\eqref{eq:record-condition} holds is therefore at least
			$p=\frac{1}{2t\ell}$.
			
			Let $Z_i$ indicate that this condition holds at the $i$th
			examination.
			If the procedure ends before $N_0$ examinations, extend
			the sequence to length $N_0$ using independent Bernoulli
			variables of mean $p$, independent of all original choices.
			The conditional success probability is at least $p$ at
			every position.
			Taking conditional expectations, the same bound holds
			when conditioning only on the earlier indicators:
			$
			\Prob[Z_i=1\mid Z_1,\ldots,Z_{i-1}]\ge p.
			$
			Since $Z_i$ takes values $0$ and $1$, we have
			\[
			\E[2^{-Z_i}\mid Z_1,\ldots,Z_{i-1}]
			=1-\frac12\Prob[Z_i=1\mid Z_1,\ldots,Z_{i-1}]
			\le1-\frac p2.
			\]
			Let $S=\sum_{i=1}^{N_0}Z_i$.
			Iterating conditional expectations gives
			$
			\E[2^{-S}]\le(1-\frac p2)^{N_0}.
			$
			Applying Markov's inequality to $2^{-S}$ with threshold
			$2^{-pN_0/2}$, and using $1-\frac{p}{2}\le e^{-p/2}$ and
			$\log 2<\frac{3}{4}$, we obtain
			\[
			\begin{aligned}
				\Prob\!\left[S\le\frac{pN_0}{2}\right]
				=\Prob\!\left[2^{-S}\ge2^{-pN_0/2}\right]
				\le2^{pN_0/2}\E[2^{-S}]
				\le2^{pN_0/2}\left(1-\frac p2\right)^{N_0}
				\le e^{-pN_0/8}.
			\end{aligned}
			\]
			
			Now suppose that $N_0$ cycles are examined without
			discarding the word.
			Each occurrence of $Z_i=1$ uses one symbol of the word,
			so $S\le r_0$.
			Since
			$
			pN_0=\frac{h}{32t^2\ell}
			=2^{11}a\ell\ge2(2a-1)=2r_0,
			$
			this event implies $S\le pN_0/2$.
			Consequently, its probability is at most
			$
			e^{-pN_0/8}=e^{-\frac{h}{256t^2\ell}}.
			$
		\end{proof}
	
		\begin{claim}\label{cl:stopped}
			Fix a word and condition on the exact choices read in a procedure that stops
			before $N_0$ examinations without discarding the word.
			For every such event of positive probability, the conditional probability that
			some bad separation agrees with the procedure is at most
			\[
				2h e^{-\frac{h}{16t\ell}}+e^{-\frac{h}{64t^2\ell^2}}.
			\]
		\end{claim}
		\begin{proof}
			Let $E_1$ denote the specified event, including the ordered list of
			examined cycles and their choices, and let $\calC_0$ be this set of cycles.
			Changing choices on other cycles does not change the procedure: at each stage
			the next cycle and the decision to stop are determined by choices already read.
			Thus $E_1$ fixes only $(\xi_C)_{C\in\calC_0}$, and the remaining choices
			are still independent and uniform conditional on $E_1$.
			All probabilities in this proof are conditional on $E_1$ alone, not on the
			existence of a bad separation.
			
			Let $W=\bigcup_{C\in\calC_0}V(C)$.
			Since fewer than $N_0$ cycles were examined, $|W|<\frac{h}{16}$, so
			$\kappa(D'-W)\ge h-|W|>\frac{15h}{16}$.
			Every remaining cycle is contained in $A_L$ or $A_R$.
			Let $V_L$ and $V_R$ be the unions of the vertex sets of these two families,
			respectively, so $V(D'-W)=V_L\mathbin{\dot\cup}V_R$.
			Both sets are determined by $E_1$.
			
			For any agreeing separation, a colour-$j$ path $P$ from a cycle in $V_L$ is represented in $L'\cup X$. Indeed, $x_P\in A_L$, and the cycle containing $P$ is disjoint from $I_L$, so there is a vertex $y_{P'}\in I_L$ with $y_{P'}x_P\in E(G)$, where $v_{P'}\in L'$. This gives the arc $v_{P'}v_P$ in $D_j$, so $v_P\notin R'$, and hence $v_P\in L'\cup X$. Similarly, if $P$ comes from a cycle in $V_R$, then there is a vertex $x_{P'}\in I_R$ with $y_Px_{P'}\in E(G)$ and $v_{P'}\in R'$. Thus $v_Pv_{P'}\in A(D_j)$, so $v_P\notin L'$, and hence $v_P\in R'\cup X$.
			
			First suppose that $|V_L|<\frac{h}{4}$.
			For every $v\in V_L\cup W$,
			\[
			|N_{D'}^+(v)\cap V_R|\ge d_U^+(v)-|V_L|-|W|>\frac{11h}{16}.
			\]
			The number of colour-$j$ paths in $V_R$ whose first vertices are outneighbours
			of $v$ is a sum of independent Bernoulli variables, with mean
			$\mu_v>\frac{11h}{16t\ell}=11\cdot2^{12}ta\ell\ge2a$.
			Its probability of being less than $a$ is at most
			$e^{-\frac{\mu_v}{8}}\le e^{-\frac{h}{16t\ell}}$.
			Since $|V_L\cup W|<h$, except with probability at most $he^{-\frac{h}{16t\ell}}$,
			every vertex of $V_L\cup W$ sends arcs to the first vertices of at least $a$
			colour-$j$ paths in $V_R$.
			For any agreeing bad separation, the last vertex of a path represented in
			$L'$ lies in $V_L\cup W$.
			One of its at least $a$ target paths in $V_R$ avoids $X$, and hence is
			represented in $R'$, a contradiction.
			
			If $|V_R|<\frac{h}{4}$, apply the same argument to inneighbours.
			Except with probability at most $he^{-\frac{h}{16t\ell}}$, the first vertex of every
			path represented in $R'$ receives arcs from the last vertices of at least $a$
			colour-$j$ paths in $V_L$.
			At least one of those paths is represented in $L'$, again a contradiction.
			
			Suppose next that $|V_L|,|V_R|\ge \frac{h}{4}$.
			Since $\frac{h}{8}=2^{13}t^2a\ell^2$ is an integer and $D'-W$ is $(\frac{h}{8})$-strongly
			connected, \Cref{lem:cutmatching} gives a matching of $\frac{h}{8}$ arcs from $V_L$ to $V_R$.
			Choose this matching using a fixed ordering, so that it depends only on
			$E_1$.
			Partition $V_L$ and $V_R$ into their original cycles.
			On each cycle in $V_L$, select the last vertex of its path if its colour is $j$,
			and no vertex otherwise; on each cycle in $V_R$, use the first vertex instead.
			The selections on different cycles are independent, and each vertex is selected
			with probability $\frac{1}{t\ell}$.
			\Cref{lem:part-selection} therefore applies to the matching.
			If $Y$ is the number of surviving arcs, it gives
			\[
			\mu=\E\!\left[Y\mid E_1\right]=\frac{h}{8t^2\ell^2}=2^{13}a\ge2a,
			\qquad
			\Prob\!\left[Y<a\mid E_1\right]\le e^{-\frac{h}{64t^2\ell^2}}.
			\]
			The surviving arcs use different cycles on each side, so they give
			vertex-disjoint arcs of $D_j$.
			If $Y\ge a$, the set $X$ cannot meet all these vertex-disjoint arcs, because
			$|X|<a$.
			An arc avoiding $X$ goes from $L'$ to $R'$, a contradiction.
			Combining the small-side and large-side bounds proves the claim.
		\end{proof}
		
		For a fixed word, different completed procedures that stop before $N_0$
		examinations without discarding the word correspond to disjoint events.
		\Cref{cl:stopped} gives the same conditional upper bound for each such
		event, so the total contribution from all procedures stopping before $N_0$ is at
		most
		$
		2he^{-\frac{h}{16t\ell}}+e^{-\frac{h}{64t^2\ell^2}}.
		$
		The only remaining possibility is that $N_0$ cycles are examined without
		discarding the word, whose probability is bounded by
		\Cref{cl:overflow}. Therefore, for each fixed word, the probability
		that the word is not discarded and some bad separation agrees with
		the procedure is at most
		\begin{equation}\label{eq:one-record}
			e^{-\frac{h}{256t^2\ell}}
			+2he^{-\frac{h}{16t\ell}}
			+e^{-\frac{h}{64t^2\ell^2}}
			\le3e^{-\frac{h}{512t^2\ell^2}}.
		\end{equation}
		
		We verify the last inequality. Since
		$
		\log(2h)
		=17\log2+2\log t+\log a+2\log\ell
		\le32ta\ell
		\le\frac{h}{32t\ell},
		$
		we have
		$
		2he^{-\frac{h}{16t\ell}}
		\le e^{-\frac{h}{32t\ell}}.
		$
		Now let
		$
		z=\frac{h}{512t^2\ell^2}=128a.
		$
		The three terms on the left of \eqref{eq:one-record} are therefore at most
		$
		e^{-2\ell z}, e^{-16t\ell z}, e^{-8z},
		$
		respectively. Each is at most $e^{-z}$, which proves
		\eqref{eq:one-record}.
		
		Finally, \Cref{cl:records} gives $3^{2a-1}$ possible words for each
		colour. Since $3\ell\le6a\le e^{2a}$ and
		$3^{2a-1}<e^{3a}$, a union bound over all words and colours gives
		\[
			\begin{aligned}
				&\Prob\!\left[E_0\text{ and some }D_j\text{ is not $a$-strongly connected}\right]\\
				&\qquad\le
				3\ell\,3^{2a-1}e^{-\frac{h}{512t^2\ell^2}}
				=3\ell\,3^{2a-1}e^{-128a}
				<e^{-123a}<\frac{1}{16}.
			\end{aligned}
	    \]
		Combining this with \eqref{eq:neighbour-failure} gives probability greater than
		$\frac{7}{8}$ that all required properties hold, since $d_0>a$.
	\end{proof}
	
    \section{Proofs of the main results}\label{sec:proofs}
	
	We use the directed version of Menger's theorem.
	
	\begin{theorem}[Menger's theorem; see~\cite{BG,BM}]
		\label{thm:menger}
		Let $D$ be a finite digraph, and let $X,Y\subseteq V(D)$ be
		disjoint. Then the maximum number of pairwise vertex-disjoint
		directed $X$--$Y$ paths equals the minimum size of a vertex set
		meeting every directed $X$--$Y$ path.
	\end{theorem}
	
	\begin{lemma}\label{lem:linkage}
		Let $D$ be a digraph, let $d\ge1$ be an integer, and let $D_0$
		be a $d$-strongly connected digraph whose vertices $v_P$ represent
		pairwise vertex-disjoint directed paths $P$ in $D$.
		Assume that $y_Px_{P'}\in A(D)$ whenever $v_Pv_{P'}\in A(D_0)$.
		Let $X,Y\subseteq V(D)$ be disjoint sets with
		$|X|=|Y|=r, 1\le r\le d,$
		and suppose that $X\cup Y$ is disjoint from every represented
		path.
		Assume further that every $x\in X$ sends an arc to the first
		vertex of at least $d$ represented paths, and every $y\in Y$
		receives an arc from the last vertex of at least $d$ represented
		paths.
		
		Then $D$ contains $r$ pairwise vertex-disjoint directed
		$X$--$Y$ paths such that, after deleting their two endpoints,
		each of these paths is obtained by concatenating one or more
		represented paths, each used in its entirety.
	\end{lemma}
	
	\begin{proof}
		Form an auxiliary digraph $D_0'$ from $D_0$ by adding the vertices of $X$
		as sources and the vertices of $Y$ as sinks.
		For $x\in X$ and $v_P\in V(D_0)$, add the arc $xv_P$ whenever
		$xx_P\in A(D)$, and for $y\in Y$ add the arc $v_Py$ whenever
		$y_Py\in A(D)$.
		No other arcs incident with $X\cup Y$ are added.
		
		Let $S\subseteq V(D_0')$ satisfy $|S|<r$, and let
		$
		S_0=S\cap V(D_0).
		$
		Since $|X|=|Y|=r>|S|$, there are vertices
		$x\in X\setminus S$ and
		$y\in Y\setminus S.$
		By assumption, $x$ has at least $d$ outneighbours in $V(D_0)$
		and $y$ has at least $d$ inneighbours in $V(D_0)$.
		After deleting $S_0$, each therefore still has at least
		$
		d-|S_0|>0
		$
		such neighbours.
		
		Moreover,
		$
		|S_0|\le |S|<r\le d.
		$
		Since $D_0$ is $d$-strongly connected, $D_0-S_0$ is strongly
		connected.
		Choose an outneighbour $v_P$ of $x$ and an inneighbour $v_Q$
		of $y$ in $D_0-S_0$.
		There is a directed $v_P$--$v_Q$ path in $D_0-S_0$.
		Adding the arcs from $x$ and to $y$ gives a directed
		$x$--$y$ path in $D_0'-S$.
		
		Thus no set of fewer than $r$ vertices meets every directed
		$X$--$Y$ path in $D_0'$.
		By \Cref{thm:menger}, $D_0'$ contains
		$r$ pairwise vertex-disjoint directed $X$--$Y$ paths.
		Because no direct arc from $X$ to $Y$ was added, each such path
		contains at least one vertex of $D_0$.
		Since vertices of $X$ are sources and vertices of $Y$ are
		sinks, its only vertices in $X\cup Y$ are its initial and final
		vertices.
		
		Finally, replace every auxiliary vertex $v_P$ occurring on
		these paths by the whole represented path $P$, and replace each
		auxiliary arc $v_Pv_{P'}$ by the corresponding joining arc
		$y_Px_{P'}$ of $D$.
		This produces $r$ directed $X$--$Y$ paths in $D$.
		They remain pairwise vertex-disjoint because the represented
		paths are pairwise vertex-disjoint and are disjoint from
		$X\cup Y$.
	\end{proof}
	
	\begin{proof}[Proof of \Cref{thm:quantitative}]
		Let $\ell=\lceil\log_2 a\rceil+1$ and $h=2^{18}k^2a\ell^2$, so $2k\ge4$, $2k\not\equiv0\pmod{3}$ and
		$\ell\le2a$.
		Let $D$ satisfy $\al(D)\le a$ and $\kappa(D)\ge r_k(a)+h$, and let $G=G_2(D)$.
		Take an inclusion-maximal family $\calC$ of vertex-disjoint copies of $C_{2k}$ in $G$,
		and let $R$ be the set of uncovered vertices.
		Maximality implies that $G[R]$ contains no $C_{2k}$, so
		\Cref{lem:remainder} gives
		$
		|R|\le r_k(a).
		$
		Since $D$ is $(r_k(a)+h)$-strongly connected, deleting $R$
		leaves a digraph that is at least
		$
		r_k(a)+h-|R|\ge h
		$
		strongly connected. Hence $D'=D-R$ is $h$-strongly connected. Let $U=V(D')$.
		
		Moreover, every vertex $v\in V(D)$ satisfies
		$
		d_D^+(v),d_D^-(v)\ge r_k(a)+h.
		$
		Deleting the vertices of $R$ removes at most $|R|$ inneighbours or
		outneighbours of $v$. Therefore, for every $v\in V(D),$
		\[
		d_U^+(v),d_U^-(v)
		\ge r_k(a)+h-|R|
		\ge h.
		\]
		In particular, $D'$ and $\calC$ are nonempty.
		
		Before making any random choices, apply \Cref{lem:forest}
		in $G[R]$ to the edgeless spanning linear forest on $R$, obtaining a path cover
		$\mathcal R$ with at most $a$ paths.
		Let $T$ be the set of endpoints of the paths in $\mathcal R$, so $T\subseteq R$ and $|T|\le2a$.
		By \Cref{lem:selection}, choose retained edges, oriented complementary paths and colours with
		the stated connectivity and neighbour properties.
		Let $M=\{e_C:C\in\calC\}$, let $V(M)$ be the set of endpoints of these edges, and let $Z=R\cup V(M)$.
	\begin{claim}\label{cl:initial}
		There is a path cover $\{P_1,\ldots,P_p\}$ of $D[Z]$, where
		$1\le p\le a$, containing one orientation of every edge of $M$.
		Every endpoint of these paths belongs to $T\cup V(M)$.
		In particular, the paths are disjoint from all complementary paths.
	\end{claim}
		\begin{proof}
			The paths of $\mathcal R$ together with the edges of $M$ form a spanning linear forest of $G[Z]$.
			Apply \Cref{lem:forest} to extend this forest to one with at most $\alpha(G[Z])\le a$
			components, without creating new endpoints or deleting any edge.
			Orient each component in either direction, which is possible because all its edges belong to $G_2(D)$.
			Its endpoints are old endpoints, lying in $T\cup V(M)$, and all complementary paths avoid $Z$.
		\end{proof}
		We call the chosen orientations of the retained edges in $M$
		the \emph{prescribed arcs}.
		\Cref{fig:whole-construction} gives a schematic illustration of the
		construction: the initial paths $P_1,\dots,P_p$ are successively joined into a directed
		cycle, which is then extended to a Hamilton cycle.
		\begin{claim}\label{cl:merging}
			Let $P_1,\ldots,P_q$ be pairwise vertex-disjoint directed paths,
			where $2\le q\le a$, such that
			$Z\subseteq\bigcup_{i=1}^q V(P_i)$ and every prescribed arc belongs
			to one of these paths.
			Suppose that all their endpoints lie in $T\cup V(M)$, and that their
			vertices outside $Z$ form a union of complementary paths of
			previously used colours, each traversed in its entirety and in
			its selected direction.
			
			For any unused colour $j$, these paths can be joined into
			$\lceil \frac{q}{2}\rceil$ pairwise vertex-disjoint directed paths
			satisfying the same containment and traversal conditions,
			with colour $j$ added to the used colours.
			Every endpoint of a resulting path is an endpoint of one of the original paths.
		\end{claim}
		\begin{proof}
			Let $r=\lfloor \frac{q}{2}\rfloor$. Choose $2r$ of the current paths and divide
			them into two families of $r$ paths each; if $q$ is odd, leave the remaining
			path unchanged. Let $X$ be the set of last vertices of the paths in the first
			family, and let $Y$ be the set of first vertices of the paths in the second
			family. Since the current paths are pairwise vertex-disjoint, $X$ and $Y$
			are disjoint and $|X|=|Y|=r$.
			
			Fix an unused colour $j$. By assumption, the current paths use only vertices
			of $Z$ and complementary paths of previously used colours. Hence every path
			in $\calP_j$ is disjoint from all current paths. Moreover, all vertices of
			$X\cup Y$ lie in $T\cup V(M)$, so \Cref{lem:selection} gives
			\[
			d_j^+(x)\ge a\quad\text{for all }x\in X,
			\qquad
			d_j^-(y)\ge a\quad\text{for all }y\in Y.
			\]
			Since $r\le q\le a$ and $D_j$ is $a$-strongly connected, we may apply
			\Cref{lem:linkage} with $d=a$. It gives $r$ pairwise vertex-disjoint
			directed $X$--$Y$ paths whose internal vertices are obtained from whole
			complementary paths of colour $j$.
			
			Use these paths to join each path in the first family to a distinct path in
			the second family. Each such concatenation is a directed path, since the two
			old paths are distinct and the interior of the connecting path is disjoint
			from every current path. The resulting paths are pairwise vertex-disjoint,
			and they are also disjoint from the unchanged path when $q$ is odd.
			
			No old vertex or arc is removed, so the new paths still cover $Z$ and contain
			every prescribed arc. Every newly added complementary path is used in its
			entirety and in its selected direction, and the endpoints of each new path are
			endpoints of the old paths. Thus all required properties are preserved. Finally,
			the number of resulting paths is
			$
			r+(q-2r)=\lceil \frac{q}{2}\rceil.
			$
		\end{proof}
		
		Starting with \Cref{cl:initial}, apply \Cref{cl:merging} while more than one path remains,
		using colours $1,2,\ldots$ in order.
		After $i$ rounds there are at most $\lceil \frac{p}{2^i}\rceil$ paths, so at most $\lceil\log_2 a\rceil=\ell-1$ rounds leave a single directed path.
		Its endpoints still lie in $T\cup V(M)$ and are distinct, since it contains a prescribed arc.
		Using the unused colour $\ell$, apply \Cref{lem:linkage} with $d=a$ and $r=1$ to join its
		last vertex to its first.
		The connecting path is internally disjoint from the current path, so their union is a simple
		directed cycle $C^*$ containing $Z$ and every prescribed arc.
		Every complementary path is used by $C^*$ either in its entirety or not at all.
		\begin{claim}\label{cl:extension}
			The cycle $C^*$ extends to a Hamilton cycle without changing the cyclic order of its existing vertices.
		\end{claim}
		\begin{proof}
			For each unused complementary path $P_C$, let $uv$ be the orientation of its retained edge $e_C$ on $C^*$.
			Replace the arc $uv$ by the other $u$--$v$ path along the original cycle $C$, directed from $u$ to $v$.
			All edges of $C$ are available in both directions in $D$, so this replacement is valid even if it
			traverses $P_C$ opposite to its previously selected direction.
			The inserted internal vertices are exactly $V(P_C)$, all absent from $C^*$, and replacements for
			different cycles have disjoint interiors because the cycles of $\calC$ are vertex-disjoint.
			All vertices of $R$ already lie on $C^*$, and every remaining vertex belongs to one of these inserted
			paths, so the resulting directed cycle is Hamiltonian.
		\end{proof}
		This proves \Cref{thm:quantitative}.
	\end{proof}
	\begin{proof}[Proof of \Cref{thm:asymptotic}]
		Let $t=\log(a+2)$ and first assume that $a$ is sufficiently large.
		Choose the least integer $k$ with $s=k-1\ge \frac{t}{2\log t}$ and $k\not\equiv0\pmod{3}$.
		At most one integer is skipped because of divisibility by $3$, so $s=\frac{t}{2\log t}+O(1)$.
		For large $a$ we have $\frac{\log a}{s}\le2\log t$ and $\frac{\log(2s)}{s}\le1$, and hence
		\[
		(2sa)^{1+\frac{1}{s}}
		=2sa\exp\left(\frac{\log a+\log(2s)}s\right)
		\le2e\,sat^2
		=O\left(\frac{at^3}{\log t}\right).
		\]
		The remaining summands in $r_k(a)$ are $O(ka)$, and the ceiling adds less than one, so
		$r_k(a)=O(\frac{at^3}{\log t})$.
		Also $2k=O(\frac{t}{\log t})$ and $\ell=O(t)$, giving
		\[
		h_k(a)=2^{18}k^2a\ell^2
		=O\left(\frac{at^4}{(\log t)^2}\right).
		\]
		Since $\frac{at^3}{\log t}=o(\frac{at^4}{(\log t)^2})$, \Cref{thm:quantitative}, together with
		$t\sim\log a$ and $\log t\sim\log\log a$, proves the
		claimed bound for all sufficiently large $a$.
		For the finitely many remaining integers $a\ge3$, use \Cref{thm:quantitative} with $k=2$ and
		enlarge the absolute constant $c$.
	\end{proof}
	\begin{proof}[Proof of \Cref{cor:epsilon}]
		For each fixed $\varepsilon>0$, $\frac{(\log a)^4}{(\log\log a)^2}=o(a^\varepsilon)$ as $a\to\infty$.
		The assertion follows from \Cref{thm:asymptotic}.
	\end{proof}
	
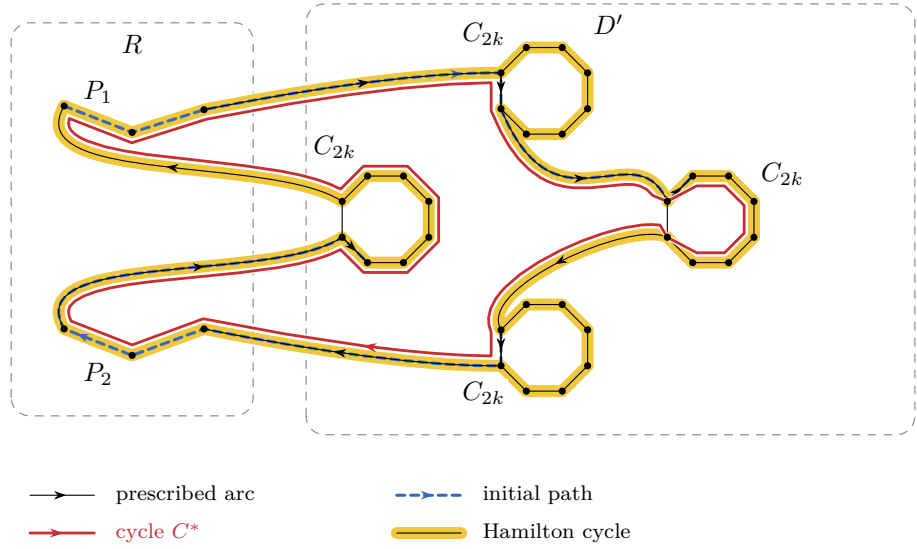
\begin{figure}[t]
	\centering
	\begingroup
	\definecolor{initblue}{RGB}{55,105,185}
	\definecolor{hamyellow}{RGB}{235,200,40}
	\definecolor{cyclecred}{RGB}{200,45,50}
	
	\begin{tikzpicture}[
		x=1cm,y=1cm,
		font=\small,
		line cap=round,
		line join=round,
		region/.style={draw=black!45, dashed, line width=.45pt, rounded corners=7pt},
		edge/.style={draw=black, line width=.42pt},
		initpath/.style={draw=initblue, line width=1.15pt, dash pattern=on 3pt off 2.2pt},
		hamcycle/.style={
			preaction={draw=white, line width=6.0pt, opacity=.65},
			draw=hamyellow!90!orange,
			line width=4.5pt,
			opacity=.88
		},
		cycleC/.style={
			preaction={draw=white, line width=2.2pt, opacity=.82},
			draw=cyclecred,
			line width=1.05pt
		},
		flow/.style={
			postaction={decorate},
			decoration={
				markings,
				mark=at position .56 with {
					\arrow{Stealth[length=1.6mm,width=1.15mm]}
				}
			}
		},
		cflow/.style={
			postaction={decorate},
			decoration={
				markings,
				mark=at position .52 with {
					\arrow{Stealth[length=1.8mm,width=1.2mm]}
				}
			}
		},
		dot/.style={circle, fill=black, inner sep=.95pt},
		lab/.style={fill=white, inner sep=1pt}
		]
		
		\draw[region] (0.0,0.7) rectangle (3.2,5.9);
		\draw[region] (3.9,0.45) rectangle (11.95,6.15);
		
		\node at (1.6,5.62) {$R$};
		\node at (7.9,5.88) {$D'$};
		
		\coordinate (a1) at (0.7,4.8);
		\coordinate (a2) at (1.6,4.45);
		\coordinate (a3) at (2.55,4.75);
		
		\coordinate (b1) at (2.55,1.85);
		\coordinate (b2) at (1.6,1.50);
		\coordinate (b3) at (0.7,1.85);
		
		\foreach \i/\cx/\cy in {
			1/7.05/5.00,
			2/9.25/3.30,
			3/4.95/3.30,
			4/7.05/1.60
		}{
			\foreach \j in {0,...,7}{
				\pgfmathsetmacro{\ang}{157.5-45*\j}
				\coordinate (q\i-\j) at ({\cx+0.62*cos(\ang)},{\cy+0.62*sin(\ang)});
			}
		}

		\draw[hamcycle]
		(a1)
		-- (a2)
		-- (a3)
		.. controls (3.55,4.95) and (5.25,5.25) .. (q1-0)
		-- (q1-1)
		-- (q1-2)
		-- (q1-3)
		-- (q1-4)
		-- (q1-5)
		-- (q1-6)
		-- (q1-7)
		.. controls (7.20,3.00) and (8.25,4.50) .. (q2-0)
		-- (q2-1)
		-- (q2-2)
		-- (q2-3)
		-- (q2-4)
		-- (q2-5)
		-- (q2-6)
		-- (q2-7)
		.. controls (8.20,3.30) and (6.15,2.40) .. (q4-0)
		-- (q4-1)
		-- (q4-2)
		-- (q4-3)
		-- (q4-4)
		-- (q4-5)
		-- (q4-6)
		-- (q4-7)
		.. controls (5.25,1.35) and (3.55,1.65) .. (b1)
		-- (b2)
		-- (b3)
		.. controls (0.10,2.90) and (3.45,2.45) .. (q3-7)
		-- (q3-6)
		-- (q3-5)
		-- (q3-4)
		-- (q3-3)
		-- (q3-2)
		-- (q3-1)
		-- (q3-0)
		.. controls (3.45,4.15) and (0.10,3.75) .. (a1);

		\draw[initpath,flow]
		(a1)
		-- (a2)
		-- (a3)
		.. controls (3.55,4.95) and (5.25,5.25) .. (q1-0)
		-- (q1-7)
		.. controls (7.20,3.00) and (8.25,4.50) .. (q2-0)
		-- (q2-1);

		\draw[initpath,flow]
		(q4-0)
		-- (q4-7)
		.. controls (5.25,1.35) and (3.55,1.65) .. (b1)
		-- (b2)
		-- (b3)
		.. controls (0.10,2.90) and (3.45,2.45) .. (q3-7)
		-- (q3-6);

		\foreach \i in {1,...,4}{
			\draw[edge]
			(q\i-0)--(q\i-1)--(q\i-2)--(q\i-3)--(q\i-4)--(q\i-5)--(q\i-6)--(q\i-7)--cycle;
		}

		\draw[edge,flow] (q1-0)--(q1-7);
		\draw[edge,flow] (q2-0)--(q2-1);
		\draw[edge,flow] (q3-7)--(q3-6);
		\draw[edge,flow] (q4-0)--(q4-7);

		\draw[edge,flow]
		(a3)
		.. controls (3.55,4.95) and (5.25,5.25) .. (q1-0);
		\draw[edge,flow]
		(q1-7)
		.. controls (7.20,3.00) and (8.25,4.50) .. (q2-0);
		\draw[edge,flow]
		(q2-7)
		.. controls (8.20,3.30) and (6.15,2.40) .. (q4-0);
		\draw[edge,flow]
		(q4-7)
		.. controls (5.25,1.35) and (3.55,1.65) .. (b1);
		\draw[edge,flow]
		(b3)
		.. controls (0.10,2.90) and (3.45,2.45) .. (q3-7);
		\draw[edge,flow]
		(q3-0)
		.. controls (3.45,4.15) and (0.10,3.75) .. (a1);

		\draw[cycleC,cflow]
		(0.76824,4.63398)
		-- (1.59552,4.31226)
		-- (2.58241,4.62391)
		-- (3.69481,4.82876)
		-- (4.46041,4.94662)
		-- (5.16358,5.03325)
		-- (5.79709,5.08726)
		-- (6.34719,5.10747)
		-- (6.34719,4.73694)
		-- (6.47257,4.46218)
		-- (6.58640,4.26656)
		-- (6.71476,4.09488)
		-- (6.84961,3.95920)
		-- (6.99962,3.85110)
		-- (7.09553,3.80136)
		-- (7.20170,3.76112)
		-- (7.31691,3.73260)
		-- (7.43013,3.71757)
		-- (7.65583,3.71721)
		-- (8.12473,3.76305)
		-- (8.22693,3.75945)
		-- (8.31369,3.74150)
		-- (8.40841,3.69250)
		-- (8.48455,3.61430)
		-- (q2-0)
		-- (9.06658,3.74281)
		-- (9.43342,3.74281)
		-- (9.69281,3.48342)
		-- (9.69281,3.11658)
		-- (9.43342,2.85719)
		-- (9.06658,2.85719)
		-- (q2-7)
		-- (8.57970,3.22425)
		-- (8.46871,3.23171)
		-- (8.34001,3.22560)
		-- (8.09873,3.18665)
		-- (7.82634,3.11347)
		-- (7.54320,3.01229)
		-- (7.26150,2.88800)
		-- (6.97337,2.73300)
		-- (6.73485,2.57444)
		-- (6.54298,2.41085)
		-- (6.47015,2.33196)
		-- (6.40843,2.24991)
		-- (6.34211,2.12276)
		-- (6.31364,2.00192)
		-- (6.31723,1.90253)
		-- (6.34719,1.80576)
		-- (6.34719,1.49253)
		-- (5.98321,1.50258)
		-- (5.57991,1.52815)
		-- (5.14956,1.56823)
		-- (4.68514,1.62316)
		-- (3.65529,1.77788)
		-- (2.53949,1.98467)
		-- (1.60124,1.63900)
		-- (0.79236,1.95357)
		-- (0.76855,2.01464)
		-- (0.75895,2.06625)
		-- (0.76287,2.11328)
		-- (0.78062,2.15369)
		-- (0.84735,2.22124)
		-- (0.96379,2.28663)
		-- (1.12786,2.34575)
		-- (1.33123,2.39611)
		-- (1.55068,2.43575)
		-- (1.82562,2.47346)
		-- (3.02049,2.59371)
		-- (3.53605,2.66254)
		-- (3.80008,2.71277)
		-- (4.01411,2.76678)
		-- (4.19585,2.82752)
		-- (4.35633,2.89975)
		-- (4.65889,2.59719)
		-- (5.24111,2.59719)
		-- (5.65281,3.00889)
		-- (5.65281,3.59111)
		-- (5.24111,4.00281)
		-- (4.65889,4.00281)
		-- (4.35640,3.70032)
		-- (4.20961,3.76774)
		-- (4.02951,3.83038)
		-- (3.81697,3.88677)
		-- (3.57126,3.93706)
		-- (3.05857,4.01305)
		-- (1.79393,4.15815)
		-- (1.49172,4.20623)
		-- (1.23991,4.26093)
		-- (1.00256,4.33809)
		-- (0.91587,4.37943)
		-- (0.84324,4.42660)
		-- (0.79537,4.47267)
		-- (0.76644,4.52305)
		-- (0.75858,4.57321)
		-- cycle;

		\node[lab] at (1.15,4.98) {$P_1$};
		\node[lab] at (1.15,1.25) {$P_2$};

		\node[lab] at (6.25,5.75) {$C_{2k}$};
		\node[lab] at (10.20,3.90) {$C_{2k}$};
		\node[lab] at (4.28,4.26) {$C_{2k}$};
		\node[lab] at (6.25,1.00) {$C_{2k}$};

		\begin{scope}[font=\scriptsize]
			\draw[edge,flow] (0.25,-0.35)--(1.10,-0.35);
			\node[anchor=west] at (1.25,-0.35) {prescribed arc};
			
			\draw[initpath,flow] (5.10,-0.35)--(5.95,-0.35);
			\node[anchor=west] at (6.10,-0.35) {initial path};
			
			\draw[cycleC,cflow] (0.25,-0.85)--(1.10,-0.85);
			\node[anchor=west, text=cyclecred] at (1.25,-0.85) {cycle $C^*$};
			
			\draw[hamcycle] (5.10,-0.85)--(5.95,-0.85);
			\draw[edge]     (5.10,-0.85)--(5.95,-0.85);
			\node[anchor=west] at (6.10,-0.85) {Hamilton cycle};
		\end{scope}
		
		\foreach \p in {a1,a2,a3,b1,b2,b3}{
			\node[dot] at (\p) {};
		}
		\foreach \i in {1,...,4}{
			\foreach \j in {0,...,7}{
				\node[dot] at (q\i-\j) {};
			}
		}
		
	\end{tikzpicture}
	\endgroup
	\caption{A schematic illustration of the construction of the Hamilton
		cycle in the proof of \Cref{thm:quantitative}.}
	\label{fig:whole-construction}
\end{figure}

\section{A counterexample to the Jackson--Ordaz pancyclicity conjecture
	for digraphs}\label{sec:pancyclicity}

We use the graph construction of Kostochka, Pudl\'ak and
R\"odl~\cite[Section~2]{KPR}, in the affine-plane form described by
Kocbek~\cite[Definition~2.2 and Theorem~4.2]{Kocbek}.
We describe the construction through a bipartite graph and prove
its required properties before constructing the following digraphs.

\begin{theorem}\label{thm:pancyclic-construction}
	For every prime $q\ge11$, there is a Hamiltonian digraph $D_q$
	with no directed $3$-cycle such that
	\[
	|V(D_q)|=4q^2(q+1),\qquad
	\kappa(D_q)=q^2(q+1),\qquad
	\al(D_q)\le8q^2+4q-4.
	\]
\end{theorem}

Let $q$ be a prime, and let $\mathbb F_q=\mathbb Z/q\mathbb Z$
denote the field of integers modulo $q$.
Define a bipartite graph $G_q$ with vertex classes
\[
	U=\{u_{x,y}:x,y\in\mathbb F_q\},\quad
	W=\{w_{a,b}:a\in\mathbb F_q\cup\{\infty\},\ b\in\mathbb F_q\}.
\]
Here $\infty$ is an additional label, not a field element.
For $a,b\in\mathbb F_q$, join $u_{x,y}$ to $w_{a,b}$ if and only
if $y=ax+b$, and join $u_{x,y}$ to $w_{\infty,b}$ if and only
if $x=b$. There are no other edges.
Thus $G_q$ is the point--line incidence graph of the affine plane
over $\mathbb F_q$, with $U$ representing points and $W$ representing
lines.

Use $0,1,\ldots,q-1$ as the representatives of the field elements.
Order $U$ lexicographically, writing $u_{x,y}\prec u_{x',y'}$
if $x<x'$, or if $x=x'$ and $y<y'$, where the comparisons are
between these integer representatives.
Define a graph $H_q$ with one vertex $z_{uw}$ for each edge
$uw\in E(G_q)$, where $u\in U$ and $w\in W$.
For $uw,u'w'\in E(G_q)$ with $u\prec u'$, define adjacency by
\begin{equation}\label{eq:affine-edge-adjacency}
	z_{uw}z_{u'w'}\in E(H_q)
	\quad\Longleftrightarrow\quad
	w\ne w'\ \text{and}\ uw'\in E(G_q).
\end{equation}
Vertices corresponding to edges with the same endpoint in $U$
are nonadjacent.

We first record the properties of $G_q$ needed below.

\begin{lemma}\label{lem:affine-base-properties}
	The graph $G_q$ has $|U|=q^2$ and $|W|=q(q+1)$.
	Every vertex in $U$ has degree $q+1$, and every vertex in $W$
	has degree $q$.
	Any two distinct vertices in $U$ have exactly one common neighbour.
	Consequently, $G_q$ contains no $4$-cycle and
	$|V(H_q)|=|E(G_q)|=q^2(q+1)$.
\end{lemma}

\begin{proof}
	The sizes of $U$ and $W$ follow from their definitions.
	For $u_{x,y}\in U$, we have
	\[
	N_{G_q}(u_{x,y})
	=\{w_{a,y-ax}:a\in\mathbb F_q\}\cup\{w_{\infty,x}\}.
	\]
	These $q+1$ neighbours are distinct.
	For $a,b\in\mathbb F_q$, the neighbours of $w_{a,b}$ are
	precisely the $q$ vertices $u_{x,ax+b}$ with $x\in\mathbb F_q$.
	The neighbours of $w_{\infty,b}$ are precisely the $q$ vertices
	$u_{b,y}$ with $y\in\mathbb F_q$.
	This proves the degree assertions.
	
	Consider distinct vertices $u_{x,y},u_{x',y'}\in U$.
	If $x=x'$, then $y\ne y'$, and $w_{\infty,x}$ is a common
	neighbour.
	No $w_{a,b}$ with $a\in\mathbb F_q$ can be adjacent to both,
	since this would give $y=ax+b=y'$.
	Thus $w_{\infty,x}$ is their unique common neighbour.
	If $x\ne x'$, there is no common neighbour of the form
	$w_{\infty,b}$.
	The equations $y=ax+b$ and $y'=ax'+b$ have the unique solution
	$a=(y-y')(x-x')^{-1}$ and $b=y-ax$, giving exactly one common
	neighbour $w_{a,b}$.
	A $4$-cycle would give two vertices in $U$ with two distinct
	common neighbours, so $G_q$ contains no $4$-cycle.
	
	Finally, summing the degrees over $U$ gives
	$|E(G_q)|=q^2(q+1)$, and $|V(H_q)|=|E(G_q)|$ by construction.
\end{proof}

The construction of $H_q$ is a special case of the ordered superline-graph construction of Kostochka, Pudl\'ak and R\"odl~\cite{KPR}; its affine-plane form was described by Kocbek~\cite{Kocbek}.
The triangle-free property and the independence-number bound below also follow from the general construction in~\cite{KPR}.
For completeness, we give direct proofs.

\begin{lemma}\label{lem:affine-triangle-free}
	The graph $H_q$ is triangle-free.
\end{lemma}

\begin{proof}
	Suppose that $z_{u_1w_1},z_{u_2w_2},z_{u_3w_3}$ form a triangle
	in $H_q$.
	By the definition of adjacency in $H_q$, the vertices $u_1,u_2,u_3$
	are pairwise distinct, and the vertices $w_1,w_2,w_3$ are also
	pairwise distinct.
	Relabel them so that $u_1\prec u_2\prec u_3$.
	Their pairwise adjacencies and~\eqref{eq:affine-edge-adjacency}
	give $u_1w_2,u_1w_3,u_2w_3\in E(G_q)$.
	Also, $u_2w_2\in E(G_q)$ by the definition of $z_{u_2w_2}$.
	Thus $u_1w_2u_2w_3u_1$ is a $4$-cycle in $G_q$, contrary
	to~\Cref{lem:affine-base-properties}.
\end{proof}

\begin{lemma}\label{lem:affine-independence}
	Every independent set in $H_q$ corresponds to a set of edges
	forming a forest in $G_q$.
	In particular, $\alpha(H_q)\le2q^2+q-1$.
\end{lemma}

\begin{proof}
	Let $I\subseteq V(H_q)$ be independent, and let $G_I$ be the
	spanning subgraph of $G_q$ with edge set
	$\{uw\in E(G_q):z_{uw}\in I\}$.
	Suppose that $G_I$ contains a cycle $C$.
	Choose the smallest vertex $u$ of $U\cap V(C)$ in the order
	$\prec$, and let $w,w'$ be its two neighbours on $C$.
	Let $u'$ be the neighbour of $w'$ on $C$ other than $u$.
	Then $u\prec u'$ and $w\ne w'$.
	Since $uw,uw',u'w'\in E(G_I)$, the vertices
	$z_{uw},z_{u'w'}\in I$ are adjacent by
	\eqref{eq:affine-edge-adjacency}, a contradiction.
	Hence $G_I$ is a forest, and
	\[
	|I|=|E(G_I)|\le |V(G_q)|-1=2q^2+q-1.\qedhere
	\]
\end{proof}

We now use four copies of $H_q$ to construct the required digraph.

\begin{proof}[Proof of \Cref{thm:pancyclic-construction}]
	Let $q\ge11$ be a prime and let $m=q^2(q+1)=|V(H_q)|$.
	Take four vertex-disjoint copies of $H_q$, with vertex sets
	$V_0,V_1,V_2,V_3$.
	Replace every edge within each copy by two opposite arcs.
	For each $i\in\{0,1,2,3\}$, add all arcs from $V_i$ to
	$V_{i+1}$, where the subscripts of the parts are read modulo $4$.
	Add no other arcs between distinct parts, and denote the
	resulting digraph by $D_q$.
	Then $|V(D_q)|=4m$.
	
	No directed $2$-cycle meets two distinct parts, so $G_2(D_q)$
	is the disjoint union of four copies of $H_q$.
	By~\Cref{lem:affine-independence}, we have
	$\al(D_q)=4\alpha(H_q)\le8q^2+4q-4$.
	
	We next prove that $\kappa(D_q)=m$.
	Let $S\subseteq V(D_q)$ with $|S|<m$.
	Every $V_i\setminus S$ is nonempty.
	Take distinct vertices $x,y\in V(D_q)\setminus S$, with
	$x\in V_i$, and choose $\eta \in\{1,2,3,4\}$ so that
	$y\in V_{i+\eta }$.
	Let $x_0=x$ and $x_{\eta}=y$, and choose
	$x_j\in V_{i+j}\setminus S$ for $1\le j<\eta $.
	These vertices are distinct, and
	$x_{j-1}x_j\in A(D_q-S)$ for every $1\le j\le \eta $.
	Hence $x_0x_1\cdots x_{\eta}$ is a directed $x$--$y$ path in
	$D_q-S$.
	Thus $D_q-S$ is strongly connected and $\kappa(D_q)\ge m$.
	Conversely, in $D_q-V_0$ no arc leaves $V_3$ for another part,
	so no vertex of $V_3$ can reach $V_1$.
	Therefore $\kappa(D_q)\le |V_0|=m$.
	
	Any directed cycle meeting distinct parts must visit all four
	parts, since arcs between parts go only from $V_i$ to $V_{i+1}$.
	Hence a directed $3$-cycle would lie in a single part, and its
	vertices would form a triangle in $H_q$, contrary
	to~\Cref{lem:affine-triangle-free}.
	
	Finally, write $V_i=\{v_{i,1},\ldots,v_{i,m}\}$.
	The $m$ vertex-disjoint directed paths
	$v_{0,j}v_{1,j}v_{2,j}v_{3,j}$, for $1\le j\le m$, together
	with the arcs $v_{3,j}v_{0,j+1}$ for $1\le j<m$ and
	$v_{3,m}v_{0,1}$, form a Hamilton cycle in $D_q$.
\end{proof}

\begin{proof}[Proof of \Cref{thm:pancyclic-obstruction}]
	For the digraphs in~\Cref{thm:pancyclic-construction},
	\[
	\frac{\kappa(D_q)}{\al(D_q)}
	\ge \frac{q^2(q+1)}{8q^2+4q-4}
	=\frac{q^2}{8q-4}
	>\frac q8.
	\]
	Given $c>0$, choose a prime $q\ge11$ with $q>8c$.
	Then $D_q$ is Hamiltonian and satisfies
	$\kappa(D_q)>c\al(D_q)$, but it has no directed $3$-cycle
	and is therefore not pancyclic.
	Moreover, $q\ge11$ gives $\kappa(D_q)\ge\al(D_q)+1$,
	so these digraphs also disprove the pancyclicity conjecture
	of Jackson and Ordaz for digraphs.
\end{proof}


\begin{thebibliography}{99}
		\raggedright
		\bibitem{AiShi}
		J.~Ai and Y.~Shi,
		A cubic refinement of Jackson's Chv\'atal--Erd\H{o}s condition
		for Hamilton cycles in digraphs,
		preprint (2026), arXiv:2606.08401.
		
		\bibitem{AS}
		N.~Alon and J.~H.~Spencer,
		\emph{The Probabilistic Method}, fourth ed., Wiley, Hoboken, NJ, 2016.
		
		\bibitem{BG}
		J.~Bang-Jensen and G.~Gutin,
		\emph{Digraphs: Theory, Algorithms and Applications}, second ed.,
		Springer, London, 2009.
		
		\bibitem{BondyP}
		J.~A.~Bondy,
		Pancyclic graphs I,
		\emph{J. Combin. Theory Ser. B} \textbf{11} (1971), 80--84.
		
		\bibitem{BondyMeta}
		J.~A.~Bondy,
		Pancyclic graphs: recent results,
		in \emph{Infinite and Finite Sets}, Vol.~I (Colloq., Keszthely, 1973),
		Colloq. Math. Soc. J\'anos Bolyai \textbf{10},
		North-Holland, Amsterdam, 1975, 181--187.
		
		\bibitem{BM}
		J.~A.~Bondy and U.~S.~R.~Murty,
		\emph{Graph Theory}, Graduate Texts in Mathematics \textbf{244},
		Springer, London, 2008.
		
		\bibitem{Caro}
		Y.~Caro,
		New results on the independence number,
		Technical Report, Tel-Aviv University, 1979.
		
		\bibitem{CFG}
		G.~Chen, R.~J.~Faudree and R.~J.~Gould,
		Intersections of longest cycles in $k$-connected graphs,
		\emph{J. Combin. Theory Ser. B} \textbf{72} (1998), 143--149.
		
		\bibitem{CE}
		V.~Chv\'atal and P.~Erd\H{o}s,
		A note on Hamiltonian circuits,
		\emph{Discrete Math.} \textbf{2} (1972), 111--113.
		
		\bibitem{DMSCE}
		N.~Dragani\'c, D.~Munh\'a Correia and B.~Sudakov,
		Chv\'atal--Erd\H{o}s condition for pancyclicity,
		\emph{J. Assoc. Math. Res.} \textbf{2} (2024), 1--14.
		
		\bibitem{DMSHam}
		N.~Dragani\'c, D.~Munh\'a Correia and B.~Sudakov,
		Pancyclicity of Hamiltonian graphs,
		\emph{J. Eur. Math. Soc.} \textbf{28} (2026), 4047--4062.
		
		\bibitem{DR}
		D.~Dubhashi and D.~Ranjan,
		Balls and bins: a study in negative dependence,
		\emph{Random Structures Algorithms} \textbf{13} (1998), 99--124.
		
		\bibitem{ErdosCycles}
		P.~Erd\H{o}s,
		Some problems in graph theory,
		in \emph{Hypergraph Seminar} (Ohio State Univ., Columbus, 1972),
		Lecture Notes in Mathematics \textbf{411},
		Springer, Berlin, 1974, 187--190.
		
		\bibitem{GM}
		T.~Gallai and A.~N.~Milgram,
		Verallgemeinerung eines graphentheoretischen Satzes von R\'edei,
		\emph{Acta Sci. Math. (Szeged)} \textbf{21} (1960), 181--186.
		
		\bibitem{GH}
		A.~Ghouila-Houri,
		Une condition suffisante d'existence d'un circuit hamiltonien,
		\emph{C. R. Acad. Sci. Paris} \textbf{251} (1960), 495--497.
		
		\bibitem{HTPancyclic}
		R.~H\"aggkvist and C.~Thomassen,
		On pancyclic digraphs,
		\emph{J. Combin. Theory Ser. B} \textbf{20} (1976), 20--40.
		
		\bibitem{Han}
		J.~Han,
		On vertex-disjoint paths in regular graphs,
		\emph{Electron. J. Combin.} \textbf{25} (2018), Paper No.~2.12, 6 pp.
		
		\bibitem{Jackson}
		B.~Jackson,
		A Chv\'atal--Erd\H{o}s condition for Hamilton cycles in digraphs,
		\emph{J. Combin. Theory Ser. B} \textbf{43} (1987), 245--252.
		
		\bibitem{JO}
		B.~Jackson and O.~Ordaz,
		Chv\'atal--Erd\H{o}s conditions for paths and cycles in graphs and digraphs:
		a survey, \emph{Discrete Math.} \textbf{84} (1990), 241--254.
		
		\bibitem{JDP}
		K.~Joag-Dev and F.~Proschan,
		Negative association of random variables with applications,
		\emph{Ann. Statist.} \textbf{11} (1983), 286--295.
		
		\bibitem{KKO}
		P.~Keevash, D.~K\"uhn and D.~Osthus,
		An exact minimum degree condition for Hamilton cycles in oriented graphs,
		\emph{J. London Math. Soc.} (2) \textbf{79} (2009), 144--166.
		
		\bibitem{KS}
		P.~Keevash and B.~Sudakov,
		Pancyclicity of Hamiltonian and highly connected graphs,
		\emph{J. Combin. Theory Ser. B} \textbf{100} (2010), 456--467.
		
		\bibitem{Kocbek}
		M.~Kocbek,
		Explicit geometric construction of Ramsey graphs,
		preprint (2025), arXiv:2507.09235.
		
		\bibitem{KPR}
		A.~Kostochka, P.~Pudl\'ak and V.~R\"odl,
		Some constructive bounds on Ramsey numbers,
		\emph{J. Combin. Theory Ser. B} \textbf{100} (2010), 439--445.
			
		\bibitem{KO}
		D.~K\"uhn and D.~Osthus,
		A survey on Hamilton cycles in directed graphs,
		\emph{European J. Combin.} \textbf{33} (2012), 750--766.
				
	   \bibitem{Letzter}
		S.~Letzter,
		Pancyclicity of highly connected graphs,
		\emph{J. Eur. Math. Soc.}, to appear,
		DOI 10.4171/JEMS/1808; arXiv:2306.12579.
							
		\bibitem{LPY}
		A.~Lo, V.~Patel and M.~A.~Y\i ld\i z,
		Cycle partitions in dense regular digraphs and oriented graphs,
		\emph{Forum Math. Sigma} \textbf{13} (2025), e79, 31 pp.
				
		\bibitem{Pikhurko}
		O.~Pikhurko,
		A note on the Tur\'an function of even cycles,
		\emph{Proc. Amer. Math. Soc.} \textbf{140} (2012), 3687--3692.
		
		\bibitem{Wei}
		V.~K.~Wei,
		A lower bound on the stability number of a simple graph,
		Bell Laboratories Technical Memorandum 81-11217-9, 1981.
			


	\end{thebibliography}
\end{document}